\documentclass[runningheads]{llncs}

\usepackage{url,enumerate, comment}
\usepackage{amssymb,graphicx}
\usepackage{pgf,tikz}
\usetikzlibrary{calc}

\newcommand{\NP}{\ensuremath{\mathsf{NP}}}
\newcommand{\etal}{et al.~}
\newcommand{\mc}{\mathcal}
\newcommand{\vc}{{\bf c}}

\newtheorem{observation}[theorem]{Observation}

\newcommand{\dpg}{dot product graph}
\newcommand{\dpd}{dot product dimension}
\newcommand{\ddpg}[1]{\ensuremath{#1}-dot product graph}
\newcommand{\ddpr}[1]{\ensuremath{#1}-dot product representation}

\newcommand{\va}{{\bf a}}

\begin{document}

\title{What Graphs are $2$-Dot Product Graphs?
\thanks{An extended abstract of this paper appeared in the proceedings of Eurocomb 2015~\cite{JPV15a}.}
}
\titlerunning{Structural Results for 2-Dot Product Graphs}
\author{
Matthew Johnson\inst{1}
\and Dani\"el Paulusma\thanks{Author supported by EPSRC (EP/G043434/1).}\inst{1}
\and Erik Jan van Leeuwen\inst{2}
\institute{Department of Computer Science, Durham University, United Kingdom\\ 
\email{\{matthew.johnson2,daniel.paulusma\}@durham.ac.uk}
\and Department of Information and Computing Sciences, Utrecht University, The Netherlands\\\email{e.j.vanleeuwen@uu.nl}}}

\maketitle

\begin{abstract}
Let $d \geq 1$ be an integer. From a set of $d$-dimensional vectors, we obtain a $d$-\dpg\ by letting each vector $\va^u$ correspond to a vertex $u$ and by adding an edge between two vertices $u$ and $v$  if and only if their dot product $\va^{u} \cdot \va^{v} \geq t$, for some fixed, positive threshold~$t$.  Dot product graphs can be used to model social networks.  Recognizing a $d$-dot product graph is known to be \NP-hard for all fixed $d\geq 2$. To understand the position of $d$-dot product graphs in the landscape of graph classes, we consider the case $d=2$, and investigate how $2$-dot product graphs relate to a number of other known graph classes including a number of well-known classes of intersection graphs.
\end{abstract}

\section{Introduction}\label{s-intro}

Consider a social network in which each individual is friends with zero or more other individuals. In a vector model of the network, an individual $u$ is described by a $d$-dimensional vector~$\va^u$ for some integer $d\geq 1$ that  expresses the extent to which~$u$ has each of a set of $d$ attributes (which might, for example, represent their hobbies, political opinions or musical tastes). Then two individuals are assumed to be friends if and only if their attributes  are ``sufficiently similar''.  There are many ways to measure similarity using a vector model (see, for example, \cite{AdamicA2003,HoffRH2002,KimL2010,LCKFG10,WattsDN2002}). In this paper, we use the {\it dot product model}: two individuals~$u$ and~$v$ are friends if and only if the dot product $\va^{u} \cdot \va^{v} \geq t$, for some fixed, positive threshold~$t$. The corresponding graph~$G$, in which each individual is a vertex and the friendship relation is described by the edge set, is called a \emph{\ddpg{d}}.  We also say that the vector model $\{ {\bf a}^{u} \mid u \in V \}$ with the threshold~$t$ is a \emph{\ddpr{d}} of~$G$.

Dot product graphs have been studied from various perspectives. In particular, the study of dot product graphs as a model for social networks was initiated in a randomized setting~\cite{Mi08,Ni07,ScheinermanT2010,YS07,YS08}, where the dot product of two vectors gives the probability that an edge occurs between the corresponding vertices. In a recent paper~\cite{JPV15}, we started the study of dot product graphs from an algorithmic perspective by considering the problems of finding a maximum independent set or a maximum clique in a $d$-dot product graph.

Fiduccia~\etal\cite{FSTZ98} initiated a study of dot product graphs from a graph-theoretic perspective. They showed that every graph on $n$ vertices and $m$ edges has a dot product representation of dimension at most the minimum of $n-1$ and $m$. They also introduced  the natural notion of the \emph{dot product dimension} of a graph, which is the smallest $d$ such that the graph has a \ddpr{d}. Graphs of dot product dimension~$1$ are easily understood and can be recognized in polynomial time: they are precisely the disjoint union of at most two threshold graphs~\cite{FSTZ98}. This situation changes for higher values of the dot product dimension. Kang and M\"{u}ller~\cite{KM12} proved that recognizing graphs of any fixed dot product dimension $d\geq 2$ is \NP-hard (whereas  membership in $\NP$ is still open for $d\geq 2$).

It is well-known (and not difficult to see) that a class of graphs is closed under vertex deletion if and only if it can be characterised by a set of minimal forbidden induced subgraphs. Hence, as $d$-dot product graphs are closed under vertex deletion, there exists a set ${\cal H}_d$ of graphs such that a graph is $d$-dot product if and only if it contains no induced subgraph isomorphic to ${\cal H}_d$.  As we will argue later, however, consistent with the aforementioned \NP-hardness of recognising $d$-dot product graphs, the set~${\cal H}_d$ has infinite size even for $d=2$, and there is little hope of determining ${\cal H}_d$ (or of finding an alternative ``easy'' characterization of the class of $d$-dot product graphs).

\subsection{Previous Work}\label{s-related}

There are few previous studies that considered graphs of small dot product dimension 
$d\geq 2$. We even lack an in-depth understanding of how $2$-dot product graphs fit within the landscape of known graph classes. For the definitions of standard graph classes, we refer the reader to the text-book of Brandst{\"a}dt, Le and Spinrad~\cite{BLS99}. 

We first note that, in some sense, there are only a limited number of \ddpg{2}s.  The \emph{speed} of a graph class is the function whose domain is the natural numbers, and, for each $n$, the value is the number of labelled graphs on $n$ vertices in the class. 
The following result will be helpful when relating various graph classes to {\dpg}s, albeit in a non-constructive manner.

\begin{theorem}[{\cite[p.~56]{Spinrad2003}}] \label{thm:2dpg:count}
For any constant $d \geq 1$, the speed of the class of \ddpg{d}s is $2^{O(n \log n)}$.
\end{theorem}

We now consider what explicit results are known, starting with the following result of Fiduccia~\etal on interval graphs.

\begin{theorem}[\cite{FSTZ98}] \label{t-interval}
Every interval graph and every caterpillar is a \ddpg{2}.
\end{theorem}

Note that not all \ddpg{2}s are interval graphs; for instance, the cycle on four vertices is a \ddpg{2} but not an interval graph.  Every interval graph is chordal and for chordal graphs the following upper bound is known ($\omega(G)$ denotes the clique number of a graph~$G$).

\begin{theorem}[\cite{FSTZ98,LC13}]\label{t-1}
Every chordal~graph~$G$ on $n$ vertices has dot product dimension at most $\min\{\omega(G)+1,n/2\}$.
\end{theorem}

Fiduccia~\etal\cite{FSTZ98} also proved that every 
tree is a $3$-dot product graph, and there exist trees of dot product dimension $3$.
This implies that
neither all outerplanar nor all planar graphs are \ddpg{2}s. Kang~\etal\cite{KLMS11} significantly strengthened this observation
and proved that every
outerplanar graph is a $3$-dot product graph and every planar graph is a $4$-dot product graph.
Moreover, Kang~\etal showed that looking at the girth of the graph  
yields strong insight into the {\dpd} of planar graphs: every
planar graph of girth at least~$5$ is a $3$-dot product graph, whereas there exist planar graphs of girth~$4$ with dot product dimension~$4$.
Recall that the girth of a graph $G$ is the length of a shortest cycle of $G$.

For cycles, we obtain the following result by combing results of Fiduccia~\etal\cite{FSTZ98} for  cycles of length not equal to~5 and Li and Chang~\cite{LC13} for the cycle of length~5. 

\begin{theorem}[\cite{FSTZ98,LC13}]\label{t-cycle}
Every cycle is a $2$-dot product graph.
\end{theorem}

A \emph{wheel} is a graph obtained from a cycle by adding a dominating vertex.  Li and Chang showed the following results for wheels.

\begin{theorem}[\cite{LC13}]\label{t-wheel}
Every wheel on six or more vertices has dot product dimension~$3$.  
\end{theorem}

Fiduccia et al.~\cite{FSTZ98} conjectured that any graph on $n$ vertices has dot product dimension at most $\lfloor\frac{n}{2}\rfloor$. Note that chordal graphs satisfy this bound due to Theorem~\ref{t-1}. Fiduccia et al.\ proved that the bound is tight for bipartite graphs. To be precise, they showed that 
every bipartite graph on $n$ vertices is an $\lfloor\frac{n}{2}\rfloor$-product graph, and the complete bipartite graph $K_{\lfloor\frac{n}{2}\rfloor,\lfloor\frac{n}{2}\rfloor}$  has dot product dimension $\lfloor\frac{n}{2}\rfloor$.
Li and Chang confirmed the conjecture of Fiduccia et al.~\cite{FSTZ98} for two more graph classes,
namely graphs of girth at least~$5$ and $P_4$-sparse graphs
(a graph is \emph{$P_4$-sparse} if every set of~$5$ vertices contains at most one~$P_4$ as an induced subgraph).

We observe that, in spite of these results, there are still many graph classes for which the relation to the class of graphs of small dot product dimension (and of dot product dimension $2$ in particular) is unclear.

\subsection{Our Results}\label{s-ours}

We provide a more complete picture of the place of \ddpg{2}s in the landscape of known graph classes.  To this end, we identify several new graph classes whose members are \ddpg{2}s. We also show that certain graph classes are neither contained in the class of \ddpg{2}s nor do they contain all \ddpg{2}s. In particular, our work provides some evidence that no well-known graph class includes all \ddpg{2}s (however, we note explicitly here that we neither claim nor conjecture this).

In Section~\ref{subsec:obs} we state several observations on \ddpg{2}s, and relate 2-dot product graphs to several further known graph classes. In the same section we present a number of lemmas that we need in the remainder of our paper.

In Section~\ref{s-com} we consider \emph{co-bipartite graphs} (complements of bipartite graphs). We prove that a complete graph minus a matching remains a \ddpg{2}. We complement this result by showing that there exist other co-bipartite graphs with dot product dimension greater than~$2$.

In Section~\ref{s-circular} we consider several classes of intersection graphs. An \emph{intersection graph} of any collection of sets has the sets as its vertex set and its edges represent when pairs of sets intersect. 
Motivated by Theorems~\ref{t-interval} and~\ref{t-1}, we further investigate these classes.
We show that, for every $k\geq 2$, the intersection graph of any set of unit caps of the $k$-dimensional unit sphere  has dot product dimension at most $k+1$.  A \emph{unit disk graph} is the intersection graph of unit-size disks in the plane. Using our general upper bound for unit spheres, we prove that unit disk graphs have dot product dimension at most $3$.  Note that this bound is sharp as the 6-wheel, which has dot product dimension~3 by Theorem~\ref{t-wheel}, is a unit disk graph (see~\cite{CS99}).  A \emph{unit circular-arc graph} is the intersection graph of unit-length arcs of a circle. We prove that not all unit circular-arc graphs have dot product dimension~$2$. Again, by using our general bound for unit spheres, we give sufficient conditions for a unit circular-arc graph to have dot product dimension~$2$.
 
In Section~\ref{s-split} we consider \emph{split graphs}; that is, graphs whose vertices can be partitioned into two sets that induce an independent set and a clique.  We show the existence of split graphs with dot product dimension greater than~2. 

\section{Observations and Supporting Lemmas} \label{subsec:obs}
Throughout, we assume that the threshold $t = 1$, unless stated otherwise.  As any $d$-dot product graph has a representation with threshold $1$ by scaling~\cite{FSTZ98}, this is no restriction.
Furthermore, by scaling (such that no two vectors have dot product exactly~$1$, see~\cite{FSTZ98}) and rotating the vectors slightly, we may assume that no two vectors are linearly dependent.
For dimension $d=2$, this allows us to define a total ordering on the vectors with respect to a reference vector.

Recall that because the class of \ddpg{2}s is closed under vertex deletion, it can be characterized by a set of forbidden induced subgraphs. However, the class of \ddpg{2}s is not well-quasi-ordered under the forbidden induced subgraph relation, that is, it has no {\it finite} set of forbidden induced subgraphs, because every wheel must be in this set of forbidden induced subgraphs by Theorem~\ref{t-wheel}. Indeed, a wheel minus a vertex is either a cycle or a fan, and thus has dot product dimension~$2$.

\begin{figure}
  \centering
  \includegraphics[scale=0.2]{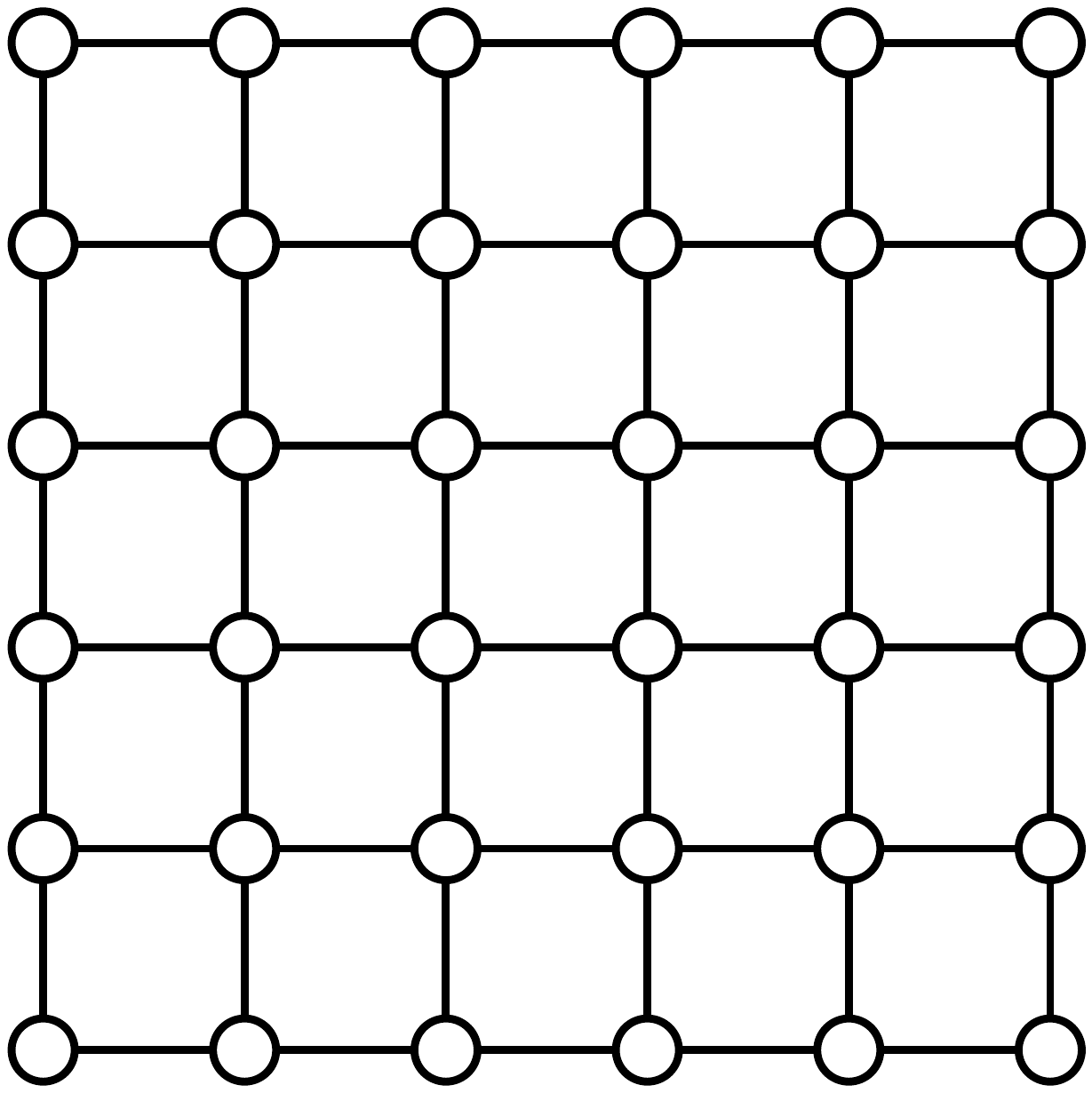}\\[-30pt]
  \caption{The grid $M_6$.}\label{f-figure}
\end{figure}

We now make the following observation to illustrate that a number of well-known graph classes are not superclasses of 2-dot product graphs. The $k\times k$ {\it grid} $M_k$ has as vertex set pairs $(i,j)$, $1 \leq i \leq k$, $1 \leq j \leq k$, and two vertices $(i,j)$ and $(i^\prime,j^\prime)$ are joined by an edge if and only if $|i-i^\prime| + |j-j^\prime| = 1$; see Figure~\ref{f-figure} for an example.
 
\begin{observation}
For each of the following graph classes, there are \ddpg{2}s that do not belong to the class:
\begin{itemize}
\item [(i)] triangle-free graphs,
\item [(ii)] claw-free graphs,
\item  [(iii)] for any graph $H$, the class of $H$-minor-free graphs,
\item  [(iv)] each of split graphs, AT-free graphs, even-hole-free graphs and odd-hole-free graphs, 
\item  [(v)] circular-arc graphs, 
\item  [(vi)] disk graphs, and
\item  [(vii)] circular perfect graphs (which is a superclass of the class of perfect graphs).
\end{itemize}
Also any grid graph $M_k$ with $k\geq 3$ is not a $2$-dot product graph. 
\end{observation}

\begin{proof}
To prove (i), we observe that the triangle is a 2-dot product graph, and, for (ii), we note that the claw has a 2-dot product representation as well (for example, take $t=3$ and vectors $(1,1)$, $(1,1)$, $(1,1)$ and $(2,2)$). To prove (iii), we observe that the class of 2-dot product graphs contains arbitrarily large cliques. Statement~(iv) follows from Theorem~\ref{t-cycle}.  Statement~(v)  is proven by taking, for example, the complete graph on four vertices and adding a pendant vertex to each vertex.  Statement~(vi) follows from taking the bi-4-wheel, which is obtained by taking a wheel on five vertices and adding a sixth vertex adjacent (only) to the four vertices of the cycle. This graph is not a disk graph, but has a 2-dot product representation: $(0,5)$, $(\frac{1}{5},2)$, $(\frac{1}{2},\frac{1}{2})$, $(\frac{1}{2},\frac{1}{2})$, $(2,\frac{1}{5})$, $(5,0)$ with $t=1$.  Because, by Theorem~\ref{t-cycle}, $C_5$ has a \ddpr{2}, we know that 2-dot product graphs are not circular perfect, which shows~(vii). Finally we note that the graph $M_3$ can easily be shown not to have a \ddpr{2} by following the proof for wheels in~\cite{LC13}.  \qed
\end{proof}

For the remainder of our paper we need some definitions and four lemmas, some of which are known already. For \ddpg{2}s, we say that a vertex $v$ is \emph{between} vertices $u$ and $w$ if $\vec{a}^{v}$ can be written as a nonnegative linear combination of $\vec{a}^{u}$ and $\vec{a}^{w}$, that is, $\vec{a}^v=\lambda_1\vec{a}^u+\lambda_2\vec{a}^w$ for some $\lambda_1,\lambda_2\geq 0$. It is important to note the geometric intuition behind this definition, as we rely on it more frequently than the algebraic definition. Namely, if $u$ and $w$ are not opposite, then $v$ is between $u$ and $w$ if $\vec{a}^{v}$ lies within the smaller of the two angles defined by $\vec{a}^{u}$ and $\vec{a}^{w}$ (see Figure~\ref{fig:between}).

\begin{figure}[tb]
\begin{center}
\begin{tikzpicture}
\draw[very thick] (0,0) -- (70:1);
\draw[very thick] (0,0) -- (90:1.5);
\draw[very thick] (0,0) -- (120:2);
\draw[very thick] (0,0) -- (135:1);
\node[above] at (90:1.5) {a};
\node[above] at (70:1) {b};
\node[above] at (120:2) {c};
\node[above left] at (135:1) {d};
\end{tikzpicture}
\end{center}
\caption{An illustration of the geometric intuition for the notion of betweenness. For example,~$a$ is between $c$ and $b$ and between $d$ and $b$, but $b$ is not between $a$ and $c$.}
\label{fig:between}
\end{figure}

\begin{lemma}[\cite{FSTZ98}]\label{lema}
Let $a,b,c$ and $d$ be vertices in a graph with a \ddpr{2} such that $ad$ and $bc$ are edges and $ac$ and $bd$ are not edges. Then $b$ and~$c$ are not both between $a$ and $d$.
\end{lemma}

\begin{lemma}[\cite{JPV15}]\label{lemb}
If $a,b$ and $c$ are vertices in a graph with a \ddpr{2}, and $c$ is between $a$ and $b$, and $ab$ is an edge but $ac$ is not an edge, then $|\va^{b}|>|\va^{c}|$.
\end{lemma}

\begin{lemma} \label{lemb2}
Let $a,b,c$ and $d$ be vertices in a graph with a \ddpr{2} such that $ab$ and $cd$ are edges, and $ac$ and $bd$ are not edges. If $c$ is between $a$ and~$b$, then $b$ is not between $c$ and $d$.
\end{lemma}

\begin{proof}
By Lemma~\ref{lemb} we have that if $c$ is between $a$ and $b$, then $|\va^{b}|>|\va^{c}|$.  However, if $b$ is between $c$ and $d$, then, using Lemma~\ref{lemb} again considering $b$, $c$ and~$d$,  we have $|\va^{c}|>|\va^{b}|$, or equivalently,  $|\va^{b}|<|\va^{c}|$, a contradiction.\qed
\end{proof}

A 4-cycle has a 2-dot product representation.  Here is one example: $(2,0)$, $(\frac{1}{2},\frac{1}{2})$, $(\frac{1}{2},1)$, $(-1,3)$ with $t=1$.  We note that if we consider the vertices in anticlockwise order from the positive $x$-axis, we find that the first and fourth vertices, and the second and third vertices, are non-adjacent.  The next lemma tells us that this is an example of a property found in all 2-dot product representations of the 4-cycle.

\begin{lemma} \label{lemc}
If $a$, $b$, $c$ and $d$ are vertices in a graph with a \ddpr{2} that induce a $4$-cycle with edges $ab$, $bc$, $cd$ and $ad$, then $\va^a$, $\va^b$, $\va^c$ and~$\va^d$ are in a half-plane whose bounding line goes through the origin.
Moreover, either $b$ and $d$ are between $a$ and $c$, or vice versa ($a$ and $c$ are between $b$ and $d$).
\end{lemma}

\begin{proof} 
By Lemma~\ref{lema}, $b$ and $c$ are not both between $a$ and $d$.   So we can assume, without loss of generality, that $b$ is not between $a$ and $d$. Then, as $a$ is adjacent to both $b$ and~$d$, 
$\va^b$ and $\va^d$ are both in the half-plane through the origin defined by $\va^a$, 
and thus the smallest angles between all pairs among $\va^a,\va^b,\va^d$ are at most $\pi$. 
Hence, either $d$ is between~$a$ and $b$ or $a$ is between $b$ and $d$.

First assume that $d$ is between $a$ and $b$. Then, by Lemma~\ref{lema}, $c$ is not between $a$ and $b$, and by Lemma~\ref{lemb2}, $a$ is not between $c$ and $d$. So $b$ and $d$ are between $a$ and $c$ and we have the required ordering.  Noting also that $a$ and $c$ are first and fourth in the ordering and have common neighbours implies that the four vectors lie in a half-plane.

Now assume that $a$ is between $b$ and $d$. Then, by Lemma~\ref{lema}, $a$ and $b$ are not both between $c$ and $d$ and $a$ and $d$ are not both between $b$ and $c$.  So $a$ and $c$ are both between $b$ and $d$ and again we have the required ordering. 
\qed
\end{proof}

\section{Co-Bipartite Graphs}\label{s-com}

In this section, we discuss some co-bipartite graphs which have a \ddpr{2} and some which do not, with the aim of understanding the relation between co-bipartite graphs and \ddpg{2}s. We first observe the following.

\begin{observation}
For any constant $d \geq 1$, there exist co-bipartite graphs that do not have a \ddpr{d}.
\end{observation}
\begin{proof}
By a result of Alekseev~\cite{Alekseev1993} (see also~\cite[Theorem~8.3]{Spinrad2003}),  the speed of co-bipartite graphs is $2^{\Theta(n^2)}$. Recalling Theorem~\ref{thm:2dpg:count}, however, the speed of \ddpg{d}s is only $2^{O(n \log n)}$. \qed
\end{proof}

As the above result is only existential, we aim for an explicit construction. In fact, we exhibit two similar classes of co-bipartite graphs  of which one has a \ddpr{2} and the other has not. First, we show that a complete graph minus a matching is a \ddpg{2}.

\begin{theorem}  \label{t-matching}
Let $G$ be a graph obtained from a complete graph by removing the edges of a matching.  Then $G$ has a \ddpr{2}.
\end{theorem}

\begin{proof}
Let $m$ be a positive integer.  Let $\{v_1, v_2, \ldots, v_m, w_1, w_2, \ldots ,w_m\}$ denote the vertex set of $K_{2m}$, and let $I_m$ denote the set of edges $\{v_iw_i \mid 1 \leq i \leq m\}$ which form a perfect matching.  We will first prove the special case of the theorem where $G=K_{2m}-I_m$.  For a nonnegative integer $k$, let $b(k)= 2^{k+1} - 1$.  For $1 \leq i \leq m$, let $\va^{v_i}  =  (1/b(i), b(i-1))$, $\va^{w_i}  =  (b(i-1), 1/b(i)). $ We show this is a \ddpr{2} for $K_{2m}-I_m$.  First consider pairs $v_i, w_i$:
\[
\va^{v_i} \cdot \va^{w_i} = \frac{2b(i-1)}{b(i)} =    
  \frac{2^{i+1} - 2}{2^{i+1} - 1} <  1.
\]
We must show that all other pairs of distinct vertices have dot product at least~1.  As $b(k) \geq 1$ for all $k \geq 0$, we have $\va^{v_i} \cdot \va^{v_j} \geq 1$ and $\va^{w_i} \cdot \va^{w_j} \geq 1$ for all $1 \leq i < j \leq m$.  Finally, for $i \neq j$,   
\[
\va^{v_i} \cdot \va^{w_j} = \frac{b(j-1)}{b(i)} + \frac{b(i-1)}{b(j)} > 1,
\]
as one of the two quotients is at least 1, and the other is positive.

For the general case, choose the largest value of $m$ such that $K_{2m}-I_m$ is an induced subgraph of $G$.  Then every vertex not in this subgraph is adjacent to every vertex other than itself.  We can obtain a \ddpr{2} of~$G$ using the representation described above for the vertices of $K_{2m}-I_m$ and by letting, for every other vertex $u$, $\va^u=(1,1)$ (and by noting that every vertex has two positive coordinates one of which is at least~$1$). \qed
\end{proof}

We need a definition: for $n \geq 2$, an {\it anti-cycle} on $n$ vertices is the complement of the cycle $C_{n}$ (replace each edge by a non-edge and vice versa) and is denoted $\overline{C_{n}}$. Note that these graphs are co-bipartite if $n$ is even.  
In the remainder of the section we will give an explicit construction of a subclass of co-bipartite graphs consisting of graphs that are not 2-dot product graphs;  this subclass will consist of sufficiently large even anti-cycles.  In fact, we first  prove that any graph that contains a sufficiently large anti-cycle --- odd or even --- is not a \ddpg{2}.

\begin{lemma}\label{l-anticycle}
If $G$ contains an induced anti-cycle on at least 7 vertices, then $G$ is not a \ddpg{2}.
\end{lemma}
\begin{proof}
Suppose that a \ddpr{2} of $G$ is given.
Let $A$ be an induced anti-cycle in $G$ whose vertices we label $v_1, v_2, \ldots, v_n$, $n \geq 7$, such that $v_i$ is non-adjacent to $v_{i+1}$, $1 \leq i \leq n-1$, $v_1$ is non-adjacent to $v_n$ and all other pairs of vertices in $A$ are adjacent.
We first show that in the \ddpr{2} of $G$ all the vertices of $A$ are in a half-plane whose bounding line goes through the origin.
Notice that $v_1$ and $v_4$ are adjacent so the angle between them is less than $\pi/2$.  Every other vertex of $A$ is adjacent to either $v_1$ or $v_4$ and so lies within less than $\pi/2$ of one or the other.  Taken together, this implies that there is a quadrant that contains no vertex of $A$.  Consider the first vertex $u$ of $A$ that is anticlockwise from the empty quadrant: all but two vertices of $A$ are adjacent to $u$ and the other two are neighbours of its neighbours.  Thus the anticlockwise angle from $\va^u$ to every other vertex is less than $\pi$.

We consider the vertices of $A$ in order of the sizes of the angles moving anticlockwise from the empty half-plane. This induces a linear ordering $\prec$.   As the vertices lie in a half-plane, if $x \prec y \prec z$, then $y$ is between $x$ and $z$; that is, if we say that one vertex is between another pair in the ordering this is equivalent to the usage of between defined earlier.

We claim that the first and last vertices in the ordering, denoted $a$ and $z$, are not adjacent.
Suppose that they are adjacent: then they are not adjacent in the cycle $\overline{A}$, and
 we can choose vertices $b$ and $y$ such that $ab$ and $yz$ are vertex-disjoint non-adjacent edges of $\overline{A}$.
Then $a$, $y$, $b$ and $z$ induce a 4-cycle in $G$ and, by Lemma~\ref{lemc}, either $a$ and $b$ must be between $y$ and $z$ or vice versa.  This contradiction proves that $a$ and $z$ cannot be adjacent.

Without loss of generality, we now assume that the first and last vertices in the ordering are, respectively, $v_3$ and $v_4$.
We claim that this implies that the first three vertices in the ordering are $v_3 \prec v_5 \prec v_1$ and that the last three are $v_6 \prec v_2 \prec v_4$.   The claim proves the lemma since it implies a contradiction: $v_5$, $v_1$, $v_6$ and $v_2$ induce a 4-cycle in $G$ and, by Lemma~\ref{lemc}, either $v_5$ and $v_6$ must be between $v_1$ and $v_2$ or vice versa.  (As an aside, we observe that if $G$ contains an anti-cycle on six vertices, $\overline{C_6}$, we can follow this proof and show that the vertices have an ordering isomorphic to $v_3 \prec v_5 \prec v_1 \prec v_6 \prec v_2 \prec v_4$.  In this case, there is no contradiction with $G$ having a \ddpr{2} as the edge $v_1v_6$ does not exist in $G$ and $v_5$, $v_1$, $v_6$ and $v_2$ do not induce a 4-cycle.)

Let us prove our claim step by step.  Let the second vertex in the ordering be $v_b$.  Suppose that $b \neq 5$.  Then $v_b$ is adjacent to $v_4$ since it is neither of its two non-neighbours.  
If $b=2$, then, since $v_3 \prec v_2 \prec v_5 \prec v_4$, the edges $v_4v_2$ and $v_5v_3$ contradict Lemma~\ref{lemb2}.
So we know $b \notin \{2,3,4,5\}$.   Let $c \equiv b+1 \bmod n$. Thus $c \notin \{3,4,5,6\}$ and $v_c$ is a non-neighbour of $v_b$ and adjacent to $v_5$.  Then the edges $v_bv_4$ and $v_5v_c$ contradict
Lemma~\ref{lema}.   Thus we must have that $b=5$.

Let the penultimate vertex in the ordering be $v_y$ and suppose that $y \neq 2$. Noting that $y \notin \{2,3,4,5\}$, we find that $v_y$ is adjacent to $v_3$.  Let $x \equiv y-1 \bmod n$ so $x \notin \{1,2,3,4\}$ and $v_x$ is adjacent to $v_2$ but not to $v_y$.
Then the edges $v_3v_y$ and $v_xv_2$ contradict
Lemma~\ref{lema}.   Thus $y=2$.

Let $v_c$ be the third vertex in the ordering and suppose that $c \neq 1$.  
Then $v_c$ is adjacent to $v_2$ since it is neither of its two non-neighbours.  
If $c=6$, then, since $v_5 \prec v_6 \prec v_1 \prec v_2$, the edges $v_5v_1$ and $v_6v_2$ contradict Lemma~\ref{lemb2}.
So we know $c \notin \{1,2,3,4,5,6\}$.   Let $d \equiv c-1 \bmod n$. Thus $d \notin \{n,1,2,3,4,5\}$ and $v_d$ is a non-neighbour of $v_c$ and adjacent to $v_1$.  Then the edges $v_cv_2$ and $v_1v_d$ contradict
Lemma~\ref{lema}.   Thus we must have that $c=1$.

Let $v_x$ be the antepenultimate vertex in the ordering and suppose that $x \neq 6$.
 Noting that $x \notin \{1,2,3,4,5,6\}$, we find that $v_x$ is adjacent to $v_5$.  Let $w \equiv x+1 \bmod n$ so $w \notin \{2,3,4,5,6,7\}$ and $v_w$ is adjacent to $v_6$ but not to $v_x$.
Then the edges $v_5v_x$ and $v_wv_6$ contradict
Lemma~\ref{lema}.   Thus $x=6$.
\qed
\end{proof}

\begin{theorem}\label{t-2n}
For $n \leq 6$, $\overline{C_{n}}$ is a \ddpg{2}. For $n > 7$, $\overline{C_{n}}$ is not a \ddpg{2}.
\end{theorem}

\begin{proof}
The cases where $n \leq 4$ are trivial. Note that $\overline{C_5} = C_5$, and thus the claim follows from Theorem~\ref{t-cycle} for $n=5$. 
We show that $\overline{C_{6}}$ is 2-dot product by displaying a representation:\\[4pt]
(5, 0), (3, 1/6),  (1/2, 1/4), (1/4, 1/2), (1/6, 3), (0, 5). \\[4pt]
(Noting the aside on the case where $G$ contains an induced $\overline{C_{6}}$ in the proof of Lemma~\ref{l-anticycle}, this is the only possible  representation up to isomorphism.)  

For $n \geq 7$, the theorem follows from Lemma~\ref{l-anticycle}.\qed
\end{proof}

\section{Unit Circular-Arc Graphs and Unit Disk Graphs}\label{s-circular}

For an integer $k\geq 2$, consider the unit sphere $S^{k}$. Then for some vector $\vc \in S^{k}$, a \emph{cap} of $S^{k}$ is the set $\{\mathbf{x} \in S^{k} \mid \vc \cdot \mathbf{x} \geq a\}$, where $a$ is a real number in $(0,1]$. We call the vector~$\vc$ the \emph{centre} of the cap, and $2\arccos a$ its \emph{angular diameter}. Observe that, given the range of $a$, the angular diameter of each cap lies in $[0,\pi)$.  
 
Fiduccia~\etal\cite{FSTZ98} considered the capture graph of caps of $S^{k}$: a graph $G$ is a capture graph if one can assign to each vertex a cap on $S^k$ so that if a pair of vertices are adjacent the centre of one cap is contained in the other and if they are not adjacent the caps are disjoint.  They showed that such a graph has dot product dimension at most~$k+1$.
This understanding of cap capture graphs was crucial in the argument that interval graphs are \ddpg{2}s (see Theorem~\ref{t-interval}) and trees are \ddpg{3}s (see Fiduccia~\etal\cite{FSTZ98}).
Kang~\etal\cite{KLMS11} studied the contact graph of caps: a graph $G$ is a contact graph if one can assign to each vertex a cap on $S^k$ so that if a pair of vertices are adjacent the caps intersect in a single point and if they are not adjacent the caps are disjoint.  They showed that such a graph has dot product dimension at most $k+2$. 
This understanding of cap contact graphs was crucial in the argument that planar graphs are \ddpg{4}s (as proved by Kang~\etal\cite{KLMS11}).

We consider \emph{unit caps}: a set of caps of $S^{k}$ is unit if all caps in the set have the same angular diameter $\theta \in [0,\pi/2)$. For unit caps we are able to prove the following general bound.

\begin{theorem} \label{thm:unitcap}
For every integer $k\geq 2$, the intersection graph of a set of unit caps of $S^{k}$ has dot product dimension at most $k+1$.
\end{theorem}

\begin{proof}
Let $\mc{C} = \{C_{1},\ldots,C_{n}\}$ denote a set of unit caps of $S^{k}$, let $\vc_{i}$ denote the centre vector of $C_{i}$, and let $\theta \in [0,\pi/2)$ denote the common angular diameter of the caps. Define $\va_{i} = \frac{1}{\sqrt{\cos \theta}}\, \vc_{i}$ and let $\mc{A} = \{\va_{1},\ldots,\va_{n}\}$. Since $\theta \in [0,\pi/2)$, $\mc{A}$ is properly defined.

Now observe that if $C_{i}$ and $C_{j}$ intersect, then the angle between $\vc_{i}$ and $\vc_{j}$ is at most $\theta$.  Hence,
\[
\va_{i} \cdot \va_{j} \geq \left(\frac{1}{\sqrt{\cos \theta}}\right)^{2}\, \cos\theta = 1.
\] 
If $C_{i}$ and $C_{j}$ do not intersect, then the angle between $\vc_{i}$ and $\vc_{j}$ is larger than $\theta$. Hence,
\[ 
\va_{i} \cdot \va_{j} < \left(\frac{1}{\sqrt{\cos \theta}}\right)^{2}\, \cos\theta = 1.
\]
It follows that the intersection graph $G$ of $\mc{C}$ is isomorphic to the dot product graph of $\mc{A}$. Since the vectors in $\mc{A}$ lie in $\mathbb{R}^{k+1}$, the dot product dimension of $G$ is at most $k+1$.\qed
\end{proof}

\begin{corollary}\label{cor:unitdisk}
Unit disk graphs have dot product dimension at most~$3$.
\end{corollary}

\begin{proof}
The intuition of our argument is that we map (through an inverse stereographic projection) the unit disks in the plane to (almost) unit caps in a region of a very large sphere. The sphere is so large that this region looks almost like a plane, but for a tiny distortion. We argue that this distortion can be made insignificantly small and thus the mapping determines a set of unit caps with the same intersecting pairs as the unit disks.

More formally, consider a set $D$ of unit disks in the plane that represent the unit disk graph. By scaling (scaling up the radius of all disks slightly, and then scaling down the entire plane), we may assume that there is an $\epsilon > 0$ such that the distance between any two disk centers is at least $1+\epsilon$ when the disks do not intersect and at most $1-\epsilon$ when the disks do intersect. Then we say that the disks have \emph{separation} $\epsilon$. By translating, we may assume that all unit disks are contained in a box of $2n \times 2n$ with its bottom-left corner at the origin. By abuse of notation, we denote this set of unit disks by $D$ as well.

We now apply the inverse of a stereographic projection. Recall that a stereographic projection considers a sphere of radius~$r$, centered at the origin, and maps a point $p \not= (0,r,0)$ on the unit sphere to the point $p'$ that is the intersection of the plane $y=-r$ with the straight line through $p$ and the north pole $(0,r,0)$ of the sphere. This mapping is bijective, so the inverse is well defined. Moreover, the mapping preserves object intersections. Note though that a stereographic projection is not isometric. However, by setting $r$ to be a number depending on $n$ and a number $\delta < \epsilon$, we can assume that the part of the sphere that projects to the $2n \times 2n$ box with its bottom-left corner at $(0,-r,0)$ is almost a plane, distorting distances between any two points in this region by at most (an additive factor) $\delta$.
Now place $D$ in the plane $y=-r$ and perform the inverse mapping on the boundaries of the unit disks in $D$.

The distortion of the mapping, unfortunately, does not map unit disks to unit caps, but to `oval caps'. Now place a unit cap of maximum possible size inside each `oval cap' and let $C$ be the resulting set of unit caps. By choosing $\delta$ significantly less than $\epsilon$, the distortion is slight and the `oval caps' closely resemble the unit caps in $C$. Moreover, this choice of $\delta$ ensures that the separation property of $D$ transfer the intersections and non-intersections to $C$. Hence, $C$ indeed represents the same graph as $D$. The result then follows from Theorem~\ref{thm:unitcap}.
\qed\end{proof}

It would seem that Theorem~\ref{thm:unitcap} also implies that all unit circular-arc graphs have dot product dimension at most~$2$. However, due to the limited angular diameter allowed in our definition of unit caps, this implication only holds if the graph has a unit circular-arc representation using unit caps of $S^{1}$. This is the case, for example, when the graph has no maximal independent set of size less than~4. 

\begin{theorem}\label{t-independentset}
If $G$ is a unit circular-arc graph with no maximal independent set of size less than~$4$, then $G$ is a \ddpg{2}.
\end{theorem}

\begin{proof}
A cap of $S^1$ is essentially equal to an arc of the circle. Our definition of unit caps limits the angular diameter of unit caps to at most $\frac{1}{2}\pi$, that is, unit arcs that each cover at most $\frac{1}{4}$ of the circle. So in order for the proof to work, we need to ensure that the unit arcs in a representation of the unit circular-arc graph each cover at most $\frac{1}{4}$ of the circle. This is guaranteed to be the case, for example, when the graph has no maximal independent set of size less than~4 (if the graph has a maximal independent set of size~4 or more, then necessarily each arc covers less than $\frac{1}{4}$ of the circle). \qed
\end{proof}

Surprisingly, the restriction on the size of a maximal independent set in Theorem~\ref{t-independentset} is not an artifact of our proof technique, but is actually needed: in Figure~\ref{fig:circ} is an example of a graph $J$ that is a unit circular-arc graph and that has dot product dimension larger than~2 (this will be shown in the proof of Theorem~\ref{l-main2dot}). Note that such an example must have triangles, due to the following proposition.

\begin{proposition}
Any triangle-free unit circular-arc graph is isomorphic to a path or a cycle. Therefore it has dot product dimension $2$.
\end{proposition}

\begin{proof}
Suppose that $G$ is a triangle-free unit circular-arc graph and is not isomorphic to a path. Since unit circular-arc graphs are claw-free, $G$ is not a tree and so contains a cycle. Let $C$ be a shortest induced cycle of $G$. Since $G$ is triangle-free, $|C| \geq 4$. Then the arcs of $C$ must cover the circle. Since $G$ is claw-free, any arc not on the cycle must intersect at least two arcs of the cycle, creating a triangle. Hence, $G$ is isomorphic to~$C$.
\qed\end{proof}

\tikzstyle{vertex}=[circle,draw=black,  minimum size=5pt, inner sep=1pt]
\tikzstyle{edge} =[draw,-,black]

\begin{figure}[tb]
\begin{center}
\begin{tikzpicture}[scale=0.9]

 \begin{scope}[xshift=5cm, scale=0.42]


  \coordinate (c2) at (0,0);

  \draw[fill=blue, red]
  ($(c2) + (0:30mm)$) arc (0:105:30mm)
  --
  ($(c2) + (105:31mm)$) arc (105:0:31mm)
  -- cycle;

  \draw[fill=blue, red]
  ($(c2) + (90:32mm)$) arc (90:195:32mm)
  --
  ($(c2) + (195:33mm)$) arc (195:90:33mm)
  -- cycle;

  \draw[fill=blue, red]
  ($(c2) + (180:30mm)$) arc (180:285:30mm)
  --
  ($(c2) + (285:31mm)$) arc (285:180:31mm)
  -- cycle;

  \draw[fill=blue, red]
  ($(c2) + (270:32mm)$) arc (270:375:32mm)
  --
  ($(c2) + (375:33mm)$) arc (375:270:33mm)
  -- cycle;

  \draw[fill=red, blue]
  ($(c2) + (127.5:34mm)$) arc (127.5:232.5:34mm)
  --
  ($(c2) + (232.5:35mm)$) arc (232.5:127.5:35mm)
  -- cycle;

  \draw[fill=red, blue]
  ($(c2) + (-52.5:34mm)$) arc (-52.5:52.5:34mm)
  --
  ($(c2) + (52.5:35mm)$) arc (52.5:-52.5:35mm)
  -- cycle;

  \draw[fill=red, blue]
  ($(c2) + (37.5:36mm)$) arc (37.5:142.5:36mm)
  --
  ($(c2) + (142.5:37mm)$) arc (142.5:37.5:37mm)
  -- cycle;

  \draw[fill=red, blue]
  ($(c2) + (217.5:36mm)$) arc (217.5:322.5:36mm)
  --
  ($(c2) + (322.5:37mm)$) arc (322.5:217.5:37mm)
  -- cycle;
\end{scope}

 \begin{scope}[xshift=0cm, scale=1]
   \foreach \pos/\name / \posn / \dist / \colour in {{(-1,0)/s/{left}/1/red}, {(0,1)/t/{above}/1/red}, {(1,0)/u/{right}/1/red}, {(0,-1)/v/{below}/1/red}, {(-1.45,1.45)/w/{left}/1/blue}, {(1.45,1.45)/x/{right}/1/blue}, {(1.45,-1.45)/y/{right}/1/blue}, {(-1.45,-1.45)/z/{left}/1/blue}}
       { \node[vertex, fill=\colour] (\name) at \pos {};
       \node [\posn=\dist] at (\name) {$\name$};
            }
             
\foreach \source/ \dest  in {s/t, t/u, u/v, v/s, w/x, x/y, y/z, z/w, s/w, s/z, t/w, t/x, u/x, u/y, v/y, v/z}
       \path[edge, black!50!white,  thick] (\source) --  (\dest);
       
\end{scope}
\end{tikzpicture}
\end{center}
\vspace{-0.5cm}
\caption{The graph $J$ and its representation as a unit circular-arc graph.}\label{fig:circ}
\end{figure}

\begin{theorem}\label{l-main2dot}
There exist unit circular-arc graphs that do not have a \ddpr{2}.
\end{theorem}

\begin{proof}
It is sufficient to show that the graph $J$ of Figure~\ref{fig:circ} does not have a \ddpr{2}. Suppose a representation exists.  By Lemma~\ref{lemc}, we can assume that $\va^s$, $\va^t$, $\va^u$ and $\va^v$ are in a half-plane and that if $\prec$ is the linear ordering given by the size of the angles from one of the half-lines bounding the half-plane, then $s \prec t \prec v \prec u$. We can include $w$, $x$, $y$ and $z$ in this ordering: if a vertex $a$ is not in the half-plane, then we say it precedes $s$ in the ordering only if $\va^a$ is in the quadrant that precedes $s$ (if, without loss of generality, $\va^s$ lies on the positive horizontal axis and $\prec$ is a clockwise ordering, then this is the top-left quadrant).

We consider the ordering of $t$, $v$, $x$ and $z$ and prove the theorem by showing that all possible orderings give a contradiction. By Lemma~\ref{lema}, $t \prec v \prec z \prec x$, is not possible.   By Lemma~\ref{lemb2}, $t \prec v \prec x \prec z$, is not possible. Thus $v$ cannot precede both $x$ and $z$, and, by the symmetry of $J$,  $x$ and $z$ cannot both precede $t$.  

Suppose that neither of $x$ and $z$ is between $t$ and $v$.  As, by  Lemma~\ref{lemb2}, $z \prec t \prec v \prec x$ is impossible,   we must have $x \prec t \prec v \prec z$.   But if $s \prec x$, then $s, x, t$ and $z$ contradict Lemma~\ref{lema}; if $x \prec s$, then $x, s, v, u$ provide the contradiction.

So one of $x$ and $z$ must be between $t$ and $v$.  Without loss of generality we can assume $t \prec x \prec v$.   If $x \prec z$, then $s, t, x$ and $z$ contradict Lemma~\ref{lema}.    If $t \prec z \prec x \prec v$, then Lemma~\ref{lemb2} is contradicted.  Finally, if $z \prec t$, then $z, t, x$ and~$v$ provide the contradiction. \qed
\end{proof}

In line with Corollary~\ref{cor:unitdisk}, we note that $J$ does indeed have a 3-dot product representation: $(2,0,1)$, $(0,2,1)$, $(-2,0,1)$, $(0,-2,1)$, $(1,1,-1)$, $(1,-1,-1)$, $(-1,-1,-1)$, $(-1,1,-1)$.

\section{Split Graphs}\label{s-split}

We prove the following result.

\begin{observation} \label{o-split}
For any constant $d \geq 1$, there exist split graphs that do not have a \ddpr{d}.
\end{observation}
\begin{proof}
By a result of Alekseev~\cite{Alekseev1993} (see also~\cite[Theorem~8.3]{Spinrad2003}), the speed of split graphs is $2^{\Theta(n^2)}$. Recalling Theorem~\ref{thm:2dpg:count}, however, the speed of \ddpg{d}s is only $2^{O(n \log n)}$.
\qed\end{proof}

\begin{figure}
\begin{center}

	\begin{tikzpicture}[scale=0.9]
\foreach \pos/ \name / \posn / \dist / \colour  in {{(0:0)/a/below/1/red}, {(90:2)/b/{above}/1/red}, {(210:2)/c/{below}/1/red}, {(330:2)/d/{below}/1/red}, {(150:2)/y/{left}/1/blue},{(320:1)/x/{left}/1/blue}}
   {     \node[vertex, fill=\colour]  (\name) at \pos {};
    	\node [\posn=\dist] at (\name) {$\name$};	}

\foreach \pos/ \name / \posn / \dist / \colour  in {{(30:2)/v/{right}/1/blue}, {(270:2)/w/{below}/1/blue},  {(200:1)/u/{left}/1/blue}, {(80:1)/z/{right}/1/blue}}
        \node[vertex, fill=\colour]  (\name) at \pos {};

\foreach \source/ \dest  in {y/b, y/c, v/b, v/d, w/c, w/d, u/a, u/c, x/a, x/d, z/a, z/b}
       \path[edge, black!50!white,  thick] (\source) --  (\dest);

\draw[edge, black!50!white, thick] (d) .. controls (270:1.4)   ..  (c);       
\draw[edge, black!50!white, thick] (b) .. controls (150:1.4)   ..  (c);
\draw[edge, black!50!white, thick] (b) .. controls (30:1.4)   ..  (d);
\draw[edge, black!50!white, thick] (a) .. controls (100:1)   ..  (b);
\draw[edge, black!50!white, thick] (a) .. controls (220:1)   ..  (c);
\draw[edge, black!50!white, thick] (a) .. controls (340:1)   ..  (d);
       
\end{tikzpicture}
\end{center}
\caption{The graph $K$.  As the vertices can be partitioned into a clique (red vertices) and an independent set (blue vertices), $K$ is a split graph.}\label{fig:split}
\end{figure} 

As the above result is only existential, we aim for an explicit construction. Let $K$ be the split graph obtained by combining a clique on four vertices $\{a,b,c,d\}$ with an independent set on six vertices; for each pair of vertices in the clique, there is a unique vertex in the independent set that is adjacent exactly to both vertices of the pair (see Figure~\ref{fig:split}).

\begin{theorem}\label{t-split2dot}
The split graph $K$ does not have a \ddpr{2}.
\end{theorem}
\begin{proof}
The vector representations  of the vertices $a$, $b$, $c$, and $d$ that form a clique must be contained in a single quadrant. As every other vertex is adjacent to the clique, the vector representations are contained within three quadrants. So there is some vertex such that there is no other vector within $\pi/2$ moving anticlockwise. Let $\prec$ be a linear ordering of the vertices according to their clockwise angle from this ``first`` vertex.

Assume without loss of generality that $a \prec b \prec c \prec d$. Let $x$ be the unique vertex of the independent set of $K$ adjacent to $a$ and $d$. By the symmetry of the graph, we can assume that $b \prec x$. (The cases $x \prec a$ and $a \prec x \prec b$ mirror the cases $d \prec x$ and $c \prec x \prec d$.) We now argue that a contradiction arises.

Let $y$ be the unique vertex of the independent set of $K$ adjacent to $b$ and $c$. We first show that $x \prec y$. Suppose instead that $y \prec x$. If $y \prec a$, then $y \prec a \prec b \prec x$ and Lemma~\ref{lemb2} is contradicted by $yb$ and $ax$. Now note that applying Lemma~\ref{lema} to $ax$ and $by$ implies that $b$ and $y$ cannot both be between $a$ and $x$. 
Thus, as we have, by assumption, that $b$ is between $a$ and $x$, we know that $y$ cannot be and so, as $a \prec y$, we must also have that $x \prec y$ as claimed.
This implies that $b \prec x \prec y$.

Now observe that Lemma~\ref{lema} applied to $by$ and $dx$ implies that $d$ and $x$ cannot both between $b$ and $y$. Then, as $b \prec x \prec y$ and $b \prec d$, it follows that $b \prec x \prec y \prec d$. But then $by$ and $xd$ contradict Lemma~\ref{lemb2}.
\qed\end{proof}

\section{Conclusions}
In this paper, we have related \ddpg{2}s to several known graph classes. 
In results of note, we shed light on the relationships between \ddpg{2}s and co-bipartite graphs, unit circular-arc graphs, and split graphs.

We are left with a number of interesting open problems.
Although we proved that unit disk graphs have {\dpd} at most~3, we do not know of an example that achieves this bound. Moreover, to prove the current bound, we considered intersection graphs of unit caps. An intriguing question is whether they are different from the capture graphs of caps studied by Fiduccia~\etal\cite{FSTZ98}. 
We 
believe, but 
have no proof,
that this is the case. Thinking more generally, bounding the {\dpd} of disk graphs, where the disks may have arbitrary sizes, is another intriguing question.

Finally, a more detailed analysis of further graph classes (such as subclasses of split graphs or unit circular-arc graphs) might increase our understanding of \ddpg{2}s. We have already shown in this work that an unbounded subclass of co-bipartite graphs has  \dpd~2. It would be interesting to further investigate such phenomena.

\subsubsection*{Acknowledgments}
We thank Till Miltzow for suggesting the easier argument for Corollary~\ref{cor:unitdisk} and an anonymous reviewer for several helpful suggestions.

\end{document}